\documentclass[smallextended,referee,envcountsect]{svjour3}
\usepackage [latin1]{inputenc}
\usepackage{amsmath,amssymb}
\usepackage{amsmath}
\usepackage{marvosym,mathtools}
\usepackage[numbers,sort&compress]{natbib}
\usepackage[colorlinks,linkcolor=blue,urlcolor=blue,citecolor=blue]{hyperref}
\usepackage{graphicx,subfig}
\usepackage{float}
\usepackage{epstopdf}
\usepackage{algorithm}
\usepackage{algorithmicx}    
\usepackage{algpseudocode}   
\usepackage{tcolorbox}

\smartqed
\usepackage{graphicx}
\journalname{}

\newcommand{\Stop}{\State \textbf{stop}}

\usepackage{fancyhdr}
\usepackage{ntheorem}
\theoremheaderfont{\bfseries\upshape}
\theorembodyfont{\upshape}
\renewtheorem{remark}{\it Remark}[section]
\renewtheorem{example}{Example}[section]

\begin{document}

\title{Primal-dual dynamical systems with closed-loop control for convex optimization in continuous and discrete time}

\author{Huan Zhang$^ {1}$ \and Xiangkai Sun$^{2}$ \and Shengjie Li$^{1}$ \and Kok Lay Teo$^{3}$}

\institute{\\Huan Zhang \at{\small zhanghwxy@163.com} \\
           \\ Xiangkai Sun  \at {\small sunxk@ctbu.edu.cn}\\
           \\Shengjie Li (\Letter) \at{\small lisj@cqu.edu.cn}\\
           \\ Kok Lay Teo \at{\small K.L.Teo@curtin.edu.au}\\\\
$^{1}$College of Mathematics and Statistics, Chongqing University, Chongqing 401331, China\\\\
$^{2}$School of Mathematics and Statistics, Chongqing Technology and Business University,
Chongqing 400067, China\\\\
$^{3}$School of Mathematical Sciences, Sunway University, Bandar Sunway, 47500 Selangor Darul Ehsan, Malaysia.}

\date{Received: date / Accepted: date}

\maketitle

\begin{abstract}
This paper develops a primal-dual dynamical system where the coefficients are designed in closed-loop manner for solving a convex optimization problem with linear equality constraints. We first introduce a ``second-order primal" + ``first-order dual'' continuous-time dynamical system, in which both the time scaling and Hessian-driven damping are governed by a feedback control of the gradient for the Lagrangian function. This system achieves the fast convergence rates for the primal-dual gap, the feasibility violation, and the objective residual along its trajectory. Subsequently, by time discretization of this system, we develop an accelerated primal-dual algorithm with a gradient-defined adaptive step size. We also obtain convergence rates for the primal-dual gap, the feasibility violation, and the objective residual. Furthermore, we provide numerical results to demonstrate the practical efficacy and superior performance of the proposed algorithm.
\end{abstract}

\section{Introduction}
\subsection{Problem description}
Let $\mathcal{H}$ and $\mathcal{G}$ be real Hilbert spaces equipped  with inner product $ \langle \cdot, \cdot \rangle$ and norm $\| \cdot \|$. In this paper, we consider the convex optimization problem with linear equality constraints of the form
\begin{eqnarray}\label{constrained}
\left\{ \begin{array}{ll}
&\mathop{\mbox{min}}\limits_{x\in\mathcal{H}}~~{f(x)}\\
&\mbox{s.t.}~~Ax=b,
\end{array}
\right.
\end{eqnarray}
where $f:\mathcal{H}\rightarrow \mathbb{R}$ is a differentiable convex function with $L$-Lipschitz continuous gradient for $L>0$, $A:\mathcal{H}\rightarrow\mathcal{G}$ is a continuous linear operator and $b\in\mathcal{G}$. The optimization problem of the form (\ref{constrained}) arises in many applications across fields such as image recovery, machine learning, and network optimization, see \cite{boyd2010,gold2014,ouyang2015,li2019} and the references therein.

The Lagrangian function associated with Problem (\ref{constrained}) is defined as
\begin{equation*}\label{PD}
\mathcal{L}(x,\lambda):= f(x)+\left \langle \lambda,Ax-b \right \rangle,
\end{equation*}
where $\lambda\in\mathcal{G}$ is the Lagrange multiplier. A pair $ (x^*,\lambda^*) \in \mathcal{H}\times \mathcal{G}$ is said to be a saddle point of the  Lagrangian function $ \mathcal{L} $ iff
\begin{equation*}
\mathcal{L}(x^*,\lambda)\leq \mathcal{L}(x^*,\lambda^*) \leq \mathcal{L}(x,\lambda^*),~~\forall (x,\lambda)\in\mathcal{H}\times\mathcal{G}.
\end{equation*}
In the sequel, the set of saddle points of $\mathcal{L}$ is denoted by $\mathcal{S}$. The set of feasible points of Problem (\ref{constrained}) is denoted by $\mathcal{F}:=\{ x\in \mathcal{H} | A x=b \}$.  For any $(x,\lambda)\in \mathcal{F} \times \mathcal{G}$, it holds that $f(x)=\mathcal{L}(x,\lambda)$. We assume that $\mathcal{S}\neq\emptyset$. Let $(x^*,\lambda^*)\in\mathcal{S}$. Then,
\begin{equation*}
(x^*,\lambda^*)\in \mathcal{S} \Leftrightarrow
\begin{cases}
0 = \nabla_x\mathcal{L}(x^*,\lambda^*)=\nabla f(x^*)+A^*\lambda^*,\\
0=\nabla_\lambda\mathcal{L}(x^*,\lambda^*)=Ax^*-b,
\end{cases}
\end{equation*}
where $A^*:\mathcal{G}\rightarrow\mathcal{H}$ denotes the adjoint operator of $A$.

To solve Problem (\ref{constrained}), in this paper, we propose the following ``second-order primal'' + ``first-order dual'' dynamical system whose damping is a feedback control of the gradient of the Lagrangian function:
\begin{equation}\label{dyn}
\begin{cases}
\ddot{x}(t)+\frac{2[\dot{\tau}(t)]^2-\tau(t)\ddot{\tau}(t)}{\tau(t)\dot{\tau}(t)}\dot{x}(t)+\frac{[\dot{\tau}(t)]^2}{\tau(t)}
\frac{d}{dt}\nabla_{x}\mathcal{L}\left(x(t),\lambda(t)\right)\\
~~~~~~~~~~~~~~~~~~~~~~+\frac{\dot{\tau}(t)(\dot{\tau}(t)+\ddot{\tau}(t))}{\tau(t)}\nabla_{x}\mathcal{L}\left(x(t),\lambda(t)\right) = 0,\\
\dot{\lambda}(t)-\dot{\tau}(t) \nabla_\lambda \mathcal{L}\left(x(t)+\frac{\tau(t)}{\dot{\tau}(t)} \dot{x}(t),\lambda(t)\right)=0,\\
\tau(t)-\frac{1}{q^q}\left( t_0+\int_{t_0}^t [\mu(s)]^{\frac{1}{q}} ds \right)^q =0,\\
[\mu(t)]^p \left\| \nabla_{x}\mathcal{L}\left(x(t),\lambda(t)\right) \right\|^{p-1}=1,
\end{cases}
\end{equation}
where $q>0$, $p \geq 1$, $t\geq t_0>0$ and $\tau:[t_0,+\infty)\rightarrow(0,+\infty)$ is the time scaling function which is non-decreasing and continuously differentiable. We obtain the fast convergence rates for the primal-dual gap, the feasibility violation, and the objective residual along the trajectory generated by System (\ref{dyn}). Subsequently, we propose an autonomous primal-dual algorithm by time discretization of System (\ref{dyn}), and give some convergence analysis.

\subsection{Motivation and related works}

Over the past decades, a series of studies have explored second-order dynamical systems in continuous and discrete time for solving the following unconstrained convex optimization problem:
\begin{eqnarray}\label{non}
\mathop{\mbox{min}}\limits_{x\in\mathcal{H}}{f(x)}.
\end{eqnarray}
Su et al. \cite{su2016} proposed the following second-order dynamical system:
\begin{equation}\label{su}
\ddot{x}(t)+ \frac{\alpha}{t} \dot{x}(t)+ \nabla f(x(t))=0,
\end{equation}
where $\alpha\geq 3$, and obtained an $O\left( \frac{1}{t^2 } \right)$ convergence rate of the objective residual along the trajectory generated by System (\ref{su}). System (\ref{su}) has further been studied in several papers, including \cite{may2017,attouch2018,attouch2019}. They showed that the objective residual converges at a rate of order $o\left( \frac{1}{t^2 } \right)$ for $\alpha> 3$, and the trajectory weakly converges to the global minimizer of the objective function. In order to improve the convergence rate,
Attouch et al. \cite{a2019} proposed System (\ref{su}) with time scaling:
\begin{equation*}
\ddot{x}(t)+ \frac{\alpha}{t} \dot{x}(t)+ \beta(t) \nabla f(x(t))=0,
\end{equation*}
where $\alpha\geq 1$ and $\beta:[t_0,+\infty)\rightarrow(0,+\infty)$ is the time scaling function which is non-decreasing and continuously differentiable. They obtained an $O\left( \frac{1}{t^2 \beta(t)} \right)$ convergence rate for the objective residual along the trajectory generated by this system, and proposed an inertial proximal algorithm by time discretization of this system.

On the other hand, primal-dual dynamical system methods have also attracted increasing interest among many researchers for solving the linear equality constrained convex optimization problem (\ref{constrained}). Over the past few years, there have been a wide variety of works devoted to ``second-order primal'' + ``second-order dual'' continuous-time dynamical systems for solving Problem (\ref{constrained}), see \cite{bot2021n,hehu2021,zeng2023,hn2023}. By time discretization of continuous-time dynamical systems, new accelerated primal-dual algorithms for solving Problem (\ref{constrained}) have been proposed in \cite{hehu2022,bot2023,zhu2024,dingliu2025}. Since the computational cost of a primal-dual method primarily arises from the sub-problem involving the primal variable, and since a first-order dynamical system is simpler and more tractable to solve than a second-order one, He et al. \cite{hehufang2022} proposed the following ``second-order primal'' + ``first-order dual'' dynamical system for solving Problem (\ref{constrained}):
\begin{equation*}
\begin{cases}
\ddot{x}(t)+\gamma\dot{x}(t)+\beta(t)\nabla_x \mathcal{L} \left(x(t),\lambda(t)\right) =0,\\
\dot{\lambda}(t)-\beta(t)\nabla_{\lambda} \mathcal{L} \left(x(t)+\delta \dot{x}(t),\lambda(t)\right)=0,
\end{cases}
\end{equation*}
where $\gamma>0$ and $\delta>0$. This system can be viewed as an extension of Polyaks heavy ball with friction system in \cite{p1964}. Note that ``second-order primal'' + ``first-order dual'' dynamical systems, as well as discrete accelerated algorithms, for solving Problem (\ref{constrained}) have been studied in \cite{2024zhu,li2025,zhufang2025,s2025}.

The present study is motivated by prior works in two domains: dynamical systems with Hessian-driven damping and those with closed-loop control. We now provide a comprehensive review of these research areas.

\subsubsection{Dynamical systems with Hessian-driven damping}

It is worth noting that the incorporation of Hessian-driven damping into the dynamical system has been a significant milestone in optimization and mechanics, as Hessian-driven damping can effectively mitigate the oscillations. Building on this foundation, many researchers use dynamical systems with Hessian-driven damping to solve Problem (\ref{non}).
Alvarez et al. \cite{a2002} first proposed the second-order dynamical system with constant viscous damping and constant Hessian-driven damping:
\begin{equation}\label{hessian}
\ddot{x}(t)+ a \dot{x}(t)+\gamma \nabla^2 f(x(t)) \dot{x}(t)+ b \nabla f(x(t))=0,
\end{equation}
where $b=1$. They obtained some new convergence properties of the solution trajectory generated by System (\ref{hessian}).
Attouch et al. \cite{attouch2016} established System (\ref{hessian}) with $a=\frac{\alpha}{t}$, and obtained the $O\left( \frac{1}{t^2} \right)$ convergence rate of the objective residual for solving Problem (\ref{non}).
To further accelerate the convergence rate, Attouch et al. \cite{a2022} proposed System (\ref{hessian}) with $a=\frac{\alpha}{t}$ and $b=\beta(t)$, and obtained the $O\left( \frac{1}{t^2 \beta(t)} \right)$ convergence rate for the objective residual.
Moreover, to derive Nesterov-type inertial algorithms via standard implicit or explicit schemes, Alecsa et al. \cite{Alecsa2021} proposed the second-order dynamical system with implicit Hessian-driven damping:
\begin{equation*}
\ddot{x}(t)+ \frac{\alpha}{t} \dot{x}(t) + \nabla f\left(x(t)+\left( \gamma+\frac{\beta}{t} \right)\dot{x}(t) \right)=0.
\end{equation*}
They showed that inertial algorithms, such as Nesterov's algorithm, can be obtained via the natural explicit discretization of this dynamical system, and they also obtained the $O\left( \frac{1}{t^2} \right)$ convergence rate of the objective residual along the trajectory generated by this system.
For more results on second order dynamical systems with Hessian-driven damping for solving Problem (\ref{non}), see \cite{bot2019,bot2021,attouch2023} and the references therein.

However, to the best of our knowledge, it appears that there exist only a few papers in the literature devoted to the study of primal-dual dynamical systems with Hessian-driven damping in both continuous and discrete time for solving Problem (\ref{constrained}). More precisely,
He et al. \cite{hetian2025} proposed the following ``second-order primal'' + ``first-order dual'' dynamical system:
\begin{equation*}
\begin{cases}
\ddot{x}(t)+\frac{\alpha}{t}\dot{x}(t)+\beta(t)\frac{d}{dt} \nabla_x \mathcal{L} \left(x(t),\lambda(t)\right) +\gamma(t)\nabla_x \mathcal{L} \left(x(t),\lambda(t)\right) =0,\\
\dot{\lambda}(t)-\eta(t)\nabla_{\lambda} \mathcal{L} \left(x(t)+\frac{t}{\alpha-1} \dot{x}(t),\lambda(t)\right)=0,
\end{cases}
\end{equation*}
where $t\geq t_0$, $\alpha>1$. They obtained an $O\left( \frac{1}{t \eta(t)} \right)$ convergence rate for the primal-dual gap, feasibility violation and the objective residual along the trajectory generated by this system, and also obtained the corresponding convergence rates for the proposed primal-dual algorithm by time discretization of this system.
Li et al. \cite{lihe2025} proposed a ``second-order primal''+``second-order dual'' dynamical system with implicit Hessian damping. They also established the fast convergence properties of the proposed dynamical system under suitable conditions.
From what was mentioned above, we see that primal-dual dynamical systems with Hessian-driven damping for solving convex optimization problems with linear equality constraints have so far received much less attention compared to other types of dynamical systems. Therefore, the first aim of this paper is to investigate such systems for solving Problem (\ref{constrained}).

\subsubsection{Dynamical systems with closed-loop control}

The damping coefficients of a closed-loop dynamical system are directly governed by its evolving trajectory, which classifies the system as autonomous. Unlike non-autonomous systems, autonomous ones are generally preferred in practice because they do not depend explicitly on the time variable $t$.
To the best of our knowledge, only a few papers in the literature have studied dynamical systems with closed-loop control for the unconstrained convex optimization problem (\ref{non}). More specifically,
Lin and Jordan \cite{lin2022} introduced a second-order dynamical system with closed-loop control as follows:
\begin{equation}\label{Lin}
\begin{cases}
\ddot{x}(t)+\frac{2[\dot{\tau}(t)]^2-\tau(t)\ddot{\tau}(t)}{\tau(t)\dot{\tau}(t)}\dot{x}(t)+\frac{[\dot{\tau}(t)]^2}{\tau(t)}
\nabla^2 f(x(t))\dot{x}(t)+\frac{\dot{\tau}(t)(\dot{\tau}(t)+\ddot{\tau}(t))}{\tau(t)}\nabla f(x(t)) = 0,\\
\tau(t)-\frac{1}{4}\left( c+\int_{0}^t [\mu(s)]^{\frac{1}{2}} ds \right)^2 =0,\\
[\mu(t)]^p \left\| \nabla f (x(t)) \right\|^{p-1}=\theta,
\end{cases}
\end{equation}
where $c>0$, $\theta\in(0,1)$, and $p \geq 1$. They obtained an $O\left( \frac{1}{t^{\frac{3p+1}{2}}}\right)$ convergence rate for the objective residual along the trajectory generated by System (\ref{Lin}), and provided two algorithmic frameworks from time discretization of this system.
Attouch et al. \cite{attouch2022} proposed an inertial dynamical system:
\begin{equation*}
\ddot{x}(t)+\mathcal{G}\left(\dot{x}(t), \nabla f(x(t)), \nabla^2 f(x(t)) \right)+\nabla f(x(t))=0,
\end{equation*}
where $t \geq t_0$ and $\mathcal{G}\left(\dot{x}(t), \nabla f(x(t)), \nabla^2 f(x(t))\right)$ acts as a closed-loop control. They analyzed the asymptotic convergence and the convergence rates of the trajectories generated by this system.
Maier et al. \cite{m2023} proposed a second-order dynamical system:
\begin{equation*}
\ddot{x}(t)+\sqrt{E(t)} \dot{x}(t)+\nabla f(x(t))=0,
\end{equation*}
where $E(t):=f(x(t))-f(x^*)+\frac{1}{2}\|\dot{x}(t)\|^2$. They obtained an $O\left(\frac{1}{t^{2-\delta}}\right)$ convergence rate for the objective residual along the trajectory generated by this system with $\delta>0$.
By using time scaling and averaging technique, Attouch et al. \cite{attouchbot2025} developed autonomous inertial dynamical systems which involve vanishing viscous damping and implicit Hessian-driven damping. They obtained an $O\left( \frac{1}{t^{1+q-\frac{1}{p}}}\right)$ convergence rate for the objective residual along the trajectory generated by this system with $q>0$, $p \geq 1$ and the weak convergence of the trajectories to optimal solutions.

We observe that there are no results on primal-dual dynamical systems with closed-loop control for solving Problem (\ref{constrained}). Hence, the second aim of this paper is to investigate such systems in both continuous and discrete time.

\subsection{Main contributions}
The contributions of this paper can be more specifically stated as follows:

\textbullet \textbf{ The continuous level:} We propose a new ``second-order primal'' + ``first-order dual'' dynamical system (\ref{dyn}) with time scaling and Hessian-driven damping. The novelty of System (\ref{dyn}) considered in this paper, when compared with the system introduced in \cite{hehufang2022}, is that System (\ref{dyn}) incorporates Hessian-driven damping, with all damping coefficients governed by the gradient of the Lagrangian function with a feedback control. We note that this is the first paper to investigate primal-dual dynamical systems with closed-loop control for solving convex optimization problems with linear equality constraints. Further, the convergence rate of the objective function residual obtained in this paper matches that reported in \cite{lin2022}, which addresses System (\ref{dyn}) for unconstrained convex optimization problems.

\textbullet \textbf{ The discrete level:} We develop an accelerated autonomous primal-dual algorithm, where the step size is adaptively defined based on gradient of the Lagrangian function. This makes our numerical scheme much easier to implement than the numerical algorithm proposed in \cite{lin2022}.
By appropriately adjusting the parameters, we show that the proposed algorithm achieves an $ O\left(\frac{1}{k^{ \frac{3p-1}{2p} }}\right) $ convergence rate for the primal-dual gap, the feasibility violation, and the objective residual. Moreover, in numerical experiments, our algorithm yields a significantly faster convergence rate and consistently achieves substantially higher accuracy than state-of-the-art methods.

\subsection{Organization}

The rest of this paper is organized as follows. In Section 2, we establish fast convergence rates for the primal-dual gap, the feasibility violation, and the objective residual along the trajectories generated by System (\ref{dyn}).  In Section 3, we propose an accelerated autonomous primal-dual algorithm for solving Problem (\ref{constrained}) and provide detailed convergence analysis. In Section 4, we present numerical experiments to illustrate the obtained results.

\section{The continuous-time dynamical system}
In this section, using the Lyapunov analysis, we establish fast convergence rates for the primal-dual gap, the feasibility violation, and the objective residual along the trajectory generated by System \textup{(\ref{dyn})} under mild assumptions on the parameters. To begin, we recall the following result, which will play an important role in the sequel.

\begin{lemma}\textup{\cite[Lemma 2.2]{attouchbot2025}}\label{lemma1}
Suppose that there exist constants $C_0>0$ and $b>a \geq 0$ such that
$$
\int_{t_0}^t [\tau(s)]^a [\mu(s)]^{-b} ds \leq C_0<+\infty, ~\forall t\geq t_0,
$$
where $\tau$ is defined in terms of $\mu$ as in System \textup{(\ref{dyn})}, namely, $\tau(t)=\frac{1}{q^q}\left( t_0+\int_{t_0}^t [\mu(s)]^{\frac{1}{q}} ds \right)^q$. Then, there exists a constant $C_1>0$ such that
$$
\tau(t) \geq C_1 (t-t_0)^{\frac{qb+1}{b-a} },~\forall t \geq t_0.
$$
\end{lemma}

Now, we investigate the convergence properties of System (\ref{dyn}).

\begin{theorem}\label{Th1}
Let $ (x ,\lambda ):\left[t_0, +\infty\right)\to \mathcal{H} \times \mathcal{G}$ be a solution of System $ (\ref{dyn})$. Suppose that $q \geq 1$ and $p\geq1$.
Then, for any $ (x^*,\lambda^*)\in \mathcal{S} $, it holds that
\begin{equation*}
\mathcal{L}(x(t),\lambda^*)-\mathcal{L}(x^*,\lambda^*)=O\left(\frac{1}{t^{\frac{2qp-p+1}{2}}}\right),~\textup{as} ~  t\to +\infty ,
\end{equation*}
\begin{equation*}
\|Ax(t)-b\|=O\left( \frac{1}{t^{\frac{2qp-p+1}{2}}}\right), ~\textup{as} ~ t\to +\infty,
\end{equation*}
\begin{equation*}
|f(x(t))-f(x^*)|=O\left( \frac{1}{t^{\frac{2qp-p+1}{2}}}\right),~\textup{as} ~ t\to +\infty,
\end{equation*}
and
\begin{equation*}
\int_{t_0}^{+\infty}[\dot{\tau}(s)]^2 \left\| \nabla_x\mathcal{L}\left(x(s),\lambda(s)\right) \right\|^2 ds <+\infty.
\end{equation*}
Furthermore, $(x(s),\lambda(s))$ is bounded over $[t_0,t]$, and the upper bound depends only on the initial condition.
\end{theorem}

\begin{proof}
Clearly, System (\ref{dyn}) is equivalent to the following system:
\begin{equation}\label{h3}
\begin{cases}
\dot{v}(t)+\dot{\tau}(t)\nabla_x \mathcal{L}\left(x(t),\lambda(t)\right) = 0,\\
\dot{x}(t)+\frac{\dot{\tau}(t)}{\tau(t)}(x(t)-v(t))+\frac{[\dot{\tau}(t)]^2}{\tau(t)}\nabla_x \mathcal{L}\left(x(t),\lambda(t)\right)=0,\\
\dot{\lambda}(t)-\dot{\tau}(t)\nabla_\lambda \mathcal{L}\left(x(t)+\frac{\tau(t)}{\dot{\tau}(t)} \dot{x}(t),\lambda(t)\right)=0,\\
\tau(t)-\frac{1}{q^q}\left( t_0+\int_{t_0}^t [\mu(s)]^{\frac{1}{q}} ds \right)^q =0,\\
[\mu(t)]^p \left\| \nabla_x \mathcal{L}\left(x(t),\lambda(t)\right) \right\|^{p-1}=1.
\end{cases}
\end{equation}
For any fixed $ (x^*,\lambda^*)\in \mathcal{S} $, we introduce the energy function $ \mathcal{E}:\left[t_0,+\infty \right) \to \mathbb{R} $ defined as follows:
\begin{equation}\label{def1}
\begin{array}{rl}
\mathcal{E}(t):=\tau(t) \left(\mathcal{L}(x(t),\lambda^*)-\mathcal{L}(x^*,\lambda^*)\right)
+\frac{1}{2}\lVert v(t)-x^*\rVert^2+\frac{1}{2}\lVert \lambda(t)-\lambda^*\rVert^2.
\end{array}
\end{equation}
Obviously, $ \mathcal{E}(t) \geq0 $  for all $t\geq t_0 $.
Note that
\begin{equation*}
\begin{aligned}
\langle \dot{v}(t),v(t)-x^* \rangle&=\langle \dot{v}(t),v(t)-x(t) \rangle+\langle \dot{v}(t),x(t)-x^* \rangle\\
&=\langle \dot{\tau}(t)\nabla_x \mathcal{L}\left(x(t),\lambda(t)\right),x(t)-v(t) \rangle-\langle \dot{\tau}(t)\nabla_x \mathcal{L}\left(x(t),\lambda(t)\right),x(t)-x^* \rangle.
\end{aligned}
\end{equation*}
This together with (\ref{def1}) gives
\begin{equation*}\label{a1}
\begin{aligned}
\dot{\mathcal{E}}(t)
=&\dot{\tau}(t)\left[\mathcal{L}\left(x(t),\lambda^*\right)-\mathcal{L}(x^*,\lambda^*)\right]
+\tau(t)\langle \nabla_x \mathcal{L}(x(t),\lambda^*),\dot{x}(t)  \rangle\\
&+\langle \dot{v}(t),v(t)-x^* \rangle+ \langle \dot{\lambda}(t),\lambda(t)-\lambda^* \rangle\\
=&\dot{\tau}(t)\left[\mathcal{L}\left(x(t),\lambda^*\right)-\mathcal{L}(x^*,\lambda^*)
-\langle \nabla_x \mathcal{L}\left(x(t),\lambda(t)\right), x(t)-x^* \rangle\right]\\
&+\tau(t)\langle \nabla_x \mathcal{L}(x(t),\lambda^*),\dot{x}(t)  \rangle
+\dot{\tau}(t) \langle \nabla_x \mathcal{L}\left(x(t),\lambda(t)\right),x(t)-v(t) \rangle\\
&+ \langle \dot{\lambda}(t),\lambda(t)-\lambda^* \rangle\\
=&\dot{\tau}(t)\left[\mathcal{L}\left(x(t),\lambda^*\right)-\mathcal{L}(x^*,\lambda^*)
-\langle \nabla_x \mathcal{L}\left(x(t),\lambda^*\right), x(t)-x^* \rangle\right]\\
&-\dot{\tau}(t)\langle \lambda(t)-\lambda^*,Ax(t)-b\rangle+\tau(t)\langle \nabla_x \mathcal{L}(x(t),\lambda(t)),\dot{x}(t)  \rangle\\
&-\tau(t)\langle \lambda(t)-\lambda^*, A\dot{x}(t)\rangle+\dot{\tau}(t) \langle \nabla_x \mathcal{L}\left(x(t),\lambda(t)\right),x(t)-v(t) \rangle\\
&+ \langle \dot{\lambda}(t),\lambda(t)-\lambda^* \rangle\\
=&\dot{\tau}(t)\left[\mathcal{L}\left(x(t),\lambda^*\right)-\mathcal{L}(x^*,\lambda^*)
-\langle \nabla_x \mathcal{L}\left(x(t),\lambda^*\right), x(t)-x^* \rangle\right]\\
&+ \langle \nabla_x \mathcal{L}\left(x(t),\lambda(t)\right),\tau(t)\dot{x}(t)+\dot{\tau}(t)(x(t)-v(t)) \rangle,
\end{aligned}
\end{equation*}
where the third equality holds because $\nabla_x \mathcal{L}\left(x(t),\lambda(t)\right)=\nabla_x \mathcal{L}\left(x(t),\lambda^*\right)+A^*(\lambda(t)-\lambda^*)$, and the last equality follows from the third equation of System (\ref{h3}) and
$$
\nabla_\lambda \mathcal{L}\left(x(t)+\frac{\tau(t)}{\dot{\tau}(t)} \dot{x}(t),\lambda(t)\right)=A\left( x(t)+\frac{\tau(t)}{\dot{\tau}(t)} \dot{x}(t) \right)-b.
$$
By using the convexity of $\mathcal{L}(\cdot,\lambda^*)$, we have
\begin{equation*}
\mathcal{L}\left(x(t),\lambda^*\right)-\mathcal{L}(x^*,\lambda^*)
-\langle \nabla_x \mathcal{L}\left(x(t),\lambda^*\right), x(t)-x^* \rangle\leq 0.
\end{equation*}
From the second equation of System (\ref{h3}), we have
\begin{equation*}
\langle \nabla_x \mathcal{L}\left(x(t),\lambda(t)\right),\tau(t)\dot{x}(t)+\dot{\tau}(t)(x(t)-v(t)) \rangle=-[\dot{\tau}(t)]^2 \left\| \nabla_x\mathcal{L}\left(x(t),\lambda(t)\right) \right\|^2.
\end{equation*}
Thus,
\begin{equation}\label{h4}
\dot{\mathcal{E}}(t)\leq -[\dot{\tau}(t)]^2 \left\| \nabla_x\mathcal{L}\left(x(t),\lambda(t)\right) \right\|^2 \leq 0.
\end{equation}
This implies $\mathcal{E}(t) \leq \mathcal{E}(t_0)$  for all $t \geq t_0$. By the definition of $\mathcal{E}(t)$, we obtain
$$
\mathcal{L}\left(x(t),\lambda^*\right)-\mathcal{L}\left(x^*,\lambda^*\right)\leq \frac{\mathcal{E}(t_0)}{\tau(t)}=O\left( \frac{1}{\tau(t)} \right),~\textup{as} ~ t\to +\infty,
$$
and $\lambda(t)$ is bounded on $[t_0,+\infty)$, and the upper bound depends only on the initial condition.
From the third equation of System (\ref{h3}), it follows that
\begin{equation*}
\begin{aligned}
\lambda(t)-\lambda(t_0)=\int_{t_0}^t \dot{\lambda}(s) ds &=\int_{t_0}^t \dot{\tau}(s)\left( Ax(s)-b+\frac{\tau(s)}{\dot{\tau}(s)}A\dot{x}(t) \right) ds\\
&=\tau(t) (Ax(t)-b)-\tau(t_0) (Ax(t_0)-b).
\end{aligned}
\end{equation*}
Then,
$$
\|\tau(t)(Ax(t)-b)\| < +\infty,
$$
and hence,
$$
\|Ax(t)-b\| \leq O\left(\frac{1}{\tau(t)}\right),~\textup{as} ~ t\to +\infty.
$$
From the definition of $\mathcal{L}$, we have
$$
|f(x(t))-f(x^*)| \leq \mathcal{L}(x(t),\lambda^*)-\mathcal{L}(x^*,\lambda^*)+\|\lambda^*\| \|Ax(t)-b\|\leq
O\left(\frac{1}{\tau(t)}\right),~\textup{as} ~ t\to +\infty.
$$

On the other hand, we deduce from (\ref{h4}) that
\begin{equation*}
\dot{\mathcal{E}}(t)+[\dot{\tau}(t)]^2 \left\| \nabla_x\mathcal{L}\left(x(t),\lambda(t)\right) \right\|^2 \leq 0.
\end{equation*}
Then,
\begin{equation}\label{h6}
\mathcal{E}(t)+\int_{t_0}^{t}[\dot{\tau}(s)]^2 \left\| \nabla_x\mathcal{L}\left(x(s),\lambda(s)\right) \right\|^2 ds \leq \mathcal{E}(t_0).
\end{equation}
Thus,
\begin{equation}\label{h5}
\int_{t_0}^{+\infty}[\dot{\tau}(s)]^2 \left\| \nabla_x\mathcal{L}\left(x(s),\lambda(s)\right) \right\|^2 ds \leq \mathcal{E}(t_0) <+\infty.
\end{equation}
From the fourth equation of System (\ref{h3}), we have
\begin{equation}\label{004}
\dot{\tau}(t)=[\tau(t)]^{\frac{q-1}{q}}[\mu(t)]^{\frac{1}{q}}.
\end{equation}
This together with the last equation of System (\ref{h3}) yields
\begin{equation}\label{010}
[\dot{\tau}(t)]^2 \left\| \nabla_x\mathcal{L}\left(x(t),\lambda(t)\right) \right\|^2=[\tau(t)]^{\frac{2q-2}{q}} [\mu(t)]^{\frac{2}{q}-\frac{2p}{p-1}}.
\end{equation}
By (\ref{h5}), (\ref{010}) and Lemma \ref{lemma1}, it follows that there exists a constant $C_1>0 $ such that
\begin{equation}\label{h7}
\tau(t) \geq C_1(t-t_0)^{\frac{2qp-p+1}{2}}.
\end{equation}
Thus,
\begin{equation*}
\mathcal{L}(x(t),\lambda^*)-\mathcal{L}(x^*,\lambda^*)=O\left(\frac{1}{t^{\frac{2qp-p+1}{2}}}\right),~\textup{as} ~  t\to +\infty ,
\end{equation*}
\begin{equation*}
\|Ax(t)-b\|=O\left( \frac{1}{t^{\frac{2qp-p+1}{2}}}\right), ~\textup{as} ~ t\to +\infty,
\end{equation*}
and
\begin{equation*}
|f(x(t))-f(x^*)|=O\left( \frac{1}{t^{\frac{2qp-p+1}{2}}}\right),~\textup{as} ~ t\to +\infty.
\end{equation*}

Now, we only need to show that $\tau(s)$ is bounded over $[t_0,t]$, and the upper bound depends only on the initial condition. In fact, by using the definition of $\mathcal{E}(t)$ and the fact that $\mathcal{E}(t) \leq \mathcal{E}(t_0)$, we obtain
$$
\frac{1}{2}\|v(t)-x^*\|^2 \leq \mathcal{E}(t_0).
$$
Thus, it follows that $v(t)$ is bounded over $[t_0,+\infty)$, and the upper bound depends only on the initial condition.
Note that
$$
\tau(t)(x(t)-x^*)-\tau(t_0)(x(t_0)-x^*)=\int_{t_0}^t \left[ \dot{\tau}(s)(x(s)-x^*)+\tau(s)\dot{x}(s) \right] ds.
$$
Then,
\begin{equation}\label{005}
\begin{aligned}
\| \tau(t)(x(t)-x^*)\|\leq& \tau(t_0)\|x(t_0)-x^*\|+\int_{t_0}^t \left\| \dot{\tau}(s)(x(s)-x^*)+\tau{(s)}\dot{x}(s) \right\| ds\\
\leq& \tau(t_0)\|x(t_0)-x^*\|+\int_{t_0}^t \| \dot{\tau}(s)(v(s)-x^*) \| ds\\
&+\int_{t_0}^t [\dot{\tau}(s)]^2 \left\| \nabla_x\mathcal{L}\left(x(s),\lambda(s)\right) \right\|ds,
\end{aligned}
\end{equation}
where the last inequality holds because of the second equation of System (\ref{h3}).
From (\ref{004}) and the last equation of System (\ref{h3}), we have
\begin{equation}\label{006}
[\dot{\tau}(t)]^2 \left\| \nabla_x\mathcal{L}\left(x(t),\lambda(t)\right) \right\|
=[\tau(t)]^{\frac{2q-2}{q}} \left\| \nabla_x\mathcal{L}\left(x(t),\lambda(t)\right) \right\|^{\frac{qp+2-2p}{qp}}.
\end{equation}
Thus, by the above inequality, (\ref{005}), $\|v(t)-x^*\| \leq \sqrt{2 \mathcal{E}(t_0)}$ and the fact that $\tau(t)$ is monotonically increasing, it follows that
\begin{equation}\label{h8}
\begin{aligned}
\|x(t)-x^*\|
\leq& \frac{\tau(t_0)\|x(t_0)-x^*\|+\sqrt{2 \mathcal{E}(t_0)}(\tau{(t)}-\tau(t_0))}{\tau{(t)}}\\
&+\frac{1}{\tau{(t)}}\int_{t_0}^t \left\| [\tau(s)]^{\frac{2q-2}{q}} \nabla_x\mathcal{L}\left(x(s),\lambda(s)\right)^{\frac{qp+2-2p}{qp}} \right\| ds\\
\leq& \|x(t_0)-x^*\|+\sqrt{2 \mathcal{E}(t_0)}\\
&+\frac{1}{\tau(t)}\int_{t_0}^t [\tau(s)]^{\frac{2q-2}{q}} \left\| \nabla_x\mathcal{L}\left(x(s),\lambda(s)\right) \right\|^{\frac{qp+2-2p}{qp}}  ds.
\end{aligned}
\end{equation}
Note that
\begin{equation}\label{001}\small
\begin{aligned}
&\frac{1}{\tau(t)}\int_{t_0}^t [\tau(s)]^{\frac{2q-2}{q}} \left\| \nabla_x\mathcal{L}\left(x(s),\lambda(s)\right) \right\|^{\frac{qp+2-2p}{qp}}  ds\\
=&\frac{1}{\tau(t)}\int_{t_0}^t [\tau(s)]^{\frac{qp-2p}{2qp+2-2p}} [\tau(s)]^{\frac{qp}{2qp+2-2p}} \left([\tau(s)]^{\frac{2q-2}{q}} \left\| \nabla_x\mathcal{L}\left(x(s),\lambda(s)\right) \right\|^{\frac{2qp+2-2p}{qp}} \right)^{\frac{qp+2-2p}{2qp+2-2p}} ds.
\end{aligned}
\end{equation}

Now, we analyze (\ref{001}) by considering the following two cases.

$\mathbf{Case~I:}$   $1\leq q < 2$ and $1 \leq p \leq \frac{2}{2-q}$.
 In this case, from H\"{o}lder inequality and the fact that $\tau(t)$ is monotonically increasing, it follows that there exists a constant $C_2 > 0$ such that
\begin{equation}\label{t2}\small
\begin{aligned}
&\frac{1}{\tau(t)}\int_{t_0}^t [\tau(s)]^{\frac{qp-2p}{2qp+2-2p}} [\tau(s)]^{\frac{qp}{2qp+2-2p}} \left([\tau(s)]^{\frac{2q-2}{q}} \left\| \nabla_x\mathcal{L}\left(x(s),\lambda(s)\right) \right\|^{\frac{2qp+2-2p}{qp}} \right)^{\frac{qp+2-2p}{2qp+2-2p}} ds\\
\leq &[\tau(t_0)]^{\frac{qp-2p}{2qp+2-2p}} \frac{1}{\tau(t)}  \left( \int_{t_0}^t \tau(s) ds\right)^{\frac{qp}{2qp+2-2p}}
\left (\int_{t_0}^t [\tau(s)]^{\frac{2q-2}{q}} \left\| \nabla_x\mathcal{L}\left(x(s),\lambda(s)\right) \right\|^{\frac{2qp+2-2p}{qp}}   ds\right)^{\frac{qp+2-2p}{2qp+2-2p}}\\
\leq& C_2 [\tau(t_0)]^{\frac{qp-2p}{2qp+2-2p}} [\tau(t)]^{-\frac{qp+2-2p}{2qp+2-2p}}
\left (\int_{t_0}^t [\tau(s)]^{\frac{2q-2}{q}} \left\| \nabla_x\mathcal{L}\left(x(s),\lambda(s)\right) \right\|^{\frac{2qp+2-2p}{qp}}   ds\right)^{\frac{qp+2-2p}{2qp+2-2p}}\\
\leq& C_2 [\tau(t_0)]^{\frac{-2}{2qp+2-2p}}
\left (\int_{t_0}^t [\tau(s)]^{\frac{2q-2}{q}} \left\| \nabla_x\mathcal{L}\left(x(s),\lambda(s)\right) \right\|^{\frac{2qp+2-2p}{qp}}   ds\right)^{\frac{qp+2-2p}{2qp+2-2p}}.
\end{aligned}
\end{equation}
From (\ref{h5}) and (\ref{006}), we have
\begin{equation}\label{003}\small
\int_{t_0}^t [\tau(s)]^{\frac{2q-2}{q}} \left\| \nabla_x\mathcal{L}\left(x(s),\lambda(s)\right) \right\|^{\frac{2qp+2-2p}{qp}}   ds=\int_{t_0}^t [\dot{\tau}(s)]^2 \left\| \nabla_x\mathcal{L}\left(x(s),\lambda(s)\right) \right\|^2  ds  \leq \mathcal{E}(t_0).
\end{equation}
Combining (\ref{h8}), (\ref{001}), (\ref{t2}) and (\ref{003}), we get
$$
\|x(t)-x^*\| \leq \|x(t_0)-x^*\|+\sqrt{ 2\mathcal{E}(t_0)}+C_2[\tau(t_0)]^{\frac{-2}{2qp+2-2p}}
[\mathcal{E}(t_0)]^{\frac{qp+2-2p}{2qp+2-2p}}.
$$

$\mathbf{Case ~II:}$   $q \geq 2$ and $ p\geq 1$. In this case,
from H\"{o}lder inequality and the fact that $\tau(t)$ is monotonically increasing, it is clear that there exists a constant $C_3 > 0$ such that
\begin{equation*}\small
\begin{aligned}
&\frac{1}{\tau(t)}\int_{t_0}^t [\tau(s)]^{\frac{qp-2p}{2qp+2-2p}} [\tau(s)]^{\frac{qp}{2qp+2-2p}} \left([\tau(s)]^{\frac{2q-2}{q}} \left\| \nabla_x\mathcal{L}\left(x(s),\lambda(s)\right) \right\|^{\frac{2qp+2-2p}{qp}} \right)^{\frac{qp+2-2p}{2qp+2-2p}} ds\\
\leq& [\tau(t)]^{\frac{-qp-2}{2qp+2-2p}}  \left( \int_{t_0}^t \tau(s) ds\right)^{\frac{qp}{2qp+2-2p}}
\left(\int_{t_0}^t [\tau(s)]^{\frac{2q-2}{q}} \left\| \nabla_x\mathcal{L}\left(x(s),\lambda(s)\right) \right\|^{\frac{2qp+2-2p}{qp}}   ds\right)^{\frac{qp+2-2p}{2qp+2-2p}}\\
\leq& C_3 [\tau(t)]^{\frac{-2}{2qp+2-2p}}
\left(\int_{t_0}^t [\tau(s)]^{\frac{2q-2}{q}} \left\| \nabla_x\mathcal{L}\left(x(s),\lambda(s)\right) \right\|^{\frac{2qp+2-2p}{qp}}   ds\right)^{\frac{qp+2-2p}{2qp+2-2p}}\\
\leq& C_3 [\tau(t_0)]^{\frac{-2}{2qp+2-2p}}
\left(\int_{t_0}^t [\tau(s)]^{\frac{2q-2}{q}} \left\| \nabla_x\mathcal{L}\left(x(s),\lambda(s)\right) \right\|^{\frac{2qp+2-2p}{qp}}   ds\right)^{\frac{qp+2-2p}{2qp+2-2p}}.
\end{aligned}
\end{equation*}
This  together with (\ref{h8}), (\ref{001}) and (\ref{003}) gives
$$
\|x(t)-x^*\| \leq \|x(t_0)-x^*\|+\sqrt{ 2\mathcal{E}(t_0)}+
C_3[\tau(t_0)]^{\frac{-2}{2qp+2-2p}}[\mathcal{E}(t_0)]^{\frac{qp+2-2p}{2qp+2-2p}}.
$$

From both cases, it follows that $x(s)$ is bounded over $[t_0, t]$, and the upper bound depends only on the initial condition.
\qed
\end{proof}

\begin{remark}
Note that Attouch et al. \cite{attouchbot2025} introduced a second-order dynamical system via closed-loop control of the gradient for solving Problem (\ref{non}). In \cite[Theorem 3.5]{attouchbot2025}, they obtained an ${o}\left(\frac{\ln(t)}{t^{pq}}\right)$ convergence rate for the objective residual along the trajectory generated by the system.
Theorem \ref{Th1} extends \cite[Theorem 3.5]{attouchbot2025} from the unconstrained optimization problem to Problem (\ref{non}), and achieves a faster convergence rate for the objective residual.
\end{remark}

Now, let us consider special cases of System (\ref{dyn}). In the special case when $q=1$, System (\ref{dyn}) collapses to
\begin{equation}\label{dy1}
\begin{cases}
\ddot{x}(t)+\frac{2[\dot{\tau}(t)]^2-\tau(t)\ddot{\tau}(t)}{\tau(t)\dot{\tau}(t)}\dot{x}(t)+\frac{[\dot{\tau}(t)]^2}{\tau(t)}
\frac{d}{dt}\nabla_{x}\mathcal{L}\left(x(t),\lambda(t)\right)\\
~~~~~~~~~~~~+\frac{\dot{\tau}(t)(\dot{\tau}(t)+\ddot{\tau}(t))}{\tau(t)}\nabla_{x}\mathcal{L}\left(x(t),\lambda(t)\right) = 0,\\
\dot{\lambda}(t)-\dot{\tau}(t)\nabla_\lambda \mathcal{L}\left(x(t)+\frac{\tau(t)}{\dot{\tau}(t)} \dot{x}(t),\lambda(t)\right)=0,\\
\dot{\tau}(t)=\mu(t),\\
[\mu(t)]^p \left\| \nabla_{x}\mathcal{L}\left(x(t),\lambda(t)\right) \right\|^{p-1}=1.
\end{cases}
\end{equation}

\begin{corollary}\label{Co1}
Let $ (x ,\lambda ):\left[t_0, +\infty\right)\to \mathcal{H} \times \mathcal{G}$ be a solution of System $ (\ref{dy1})$. Suppose that $p\geq1$.
Then, for any $ (x^*,\lambda^*)\in \mathcal{S} $, it holds that
\begin{equation*}
\mathcal{L}(x(t),\lambda^*)-\mathcal{L}(x^*,\lambda^*)=O\left(\frac{1}{t^{\frac{p+1}{2}}}\right),~\textup{as} ~  t\to +\infty ,
\end{equation*}
\begin{equation*}
\|Ax(t)-b\|=O\left( \frac{1}{t^{\frac{p+1}{2}}}\right),~\textup{as} ~ t\to +\infty,
\end{equation*}
\begin{equation*}
|f(x(t))-f(x^*)|=O\left( \frac{1}{t^{\frac{p+1}{2}}}\right),~\textup{as} ~ t\to +\infty,
\end{equation*}
and
\begin{equation*}
\int_{t_0}^{+\infty}[\dot{\tau}(s)]^2 \left\| \nabla_x\mathcal{L}\left(x(s),\lambda(s)\right) \right\|^2 ds <+\infty.
\end{equation*}
Furthermore, $(x(s),\lambda(s))$ is bounded over $[t_0,t]$, and the upper bound depends only on the initial condition.
\end{corollary}

In the special case when $q=2$, System (\ref{dyn}) collapses to the following system:
\begin{equation}\label{dy2}
\begin{cases}
\ddot{x}(t)+\frac{2[\dot{\tau}(t)]^2-\tau(t)\ddot{\tau}(t)}{\tau(t)\dot{\tau}(t)}\dot{x}(t)+\frac{[\dot{\tau}(t)]^2}{\tau(t)}
\frac{d}{dt}\nabla_{x}\mathcal{L}\left(x(t),\lambda(t)\right)\\
~~~~~~~~~~~~~+\frac{\dot{\tau}(t)(\dot{\tau}(t)+\ddot{\tau}(t))}{\tau(t)}\nabla_{x}\mathcal{L}\left(x(t),\lambda(t)\right) = 0,\\
\dot{\lambda}(t)-\dot{\tau}(t)\nabla_\lambda \mathcal{L}\left(x(t)+\frac{\tau(t)}{\dot{\tau}(t)} \dot{x}(t),\lambda(t)\right)=0,\\
\tau(t)-\frac{1}{4}\left( t_0+\int_{t_0}^t [\mu(s)]^{\frac{1}{2}} ds \right)^2 =0,\\
[\mu(t)]^p \left\| \nabla_{x}\mathcal{L}\left(x(t),\lambda(t)\right) \right\|^{p-1}=1.
\end{cases}
\end{equation}

\begin{corollary}\label{Co2}
Let $ (x ,\lambda ):\left[t_0, +\infty\right)\to \mathcal{H} \times \mathcal{G}$ be a solution of System $ (\ref{dy2})$. Suppose that $p\geq1$.
Then, for any $ (x^*,\lambda^*)\in \mathcal{S} $, it holds that
\begin{equation*}
\mathcal{L}(x(t),\lambda^*)-\mathcal{L}(x^*,\lambda^*)=O\left(\frac{1}{t^{\frac{3p+1}{2}}}\right),~\textup{as} ~  t\to +\infty ,
\end{equation*}
\begin{equation*}
\|Ax(t)-b\|=O\left( \frac{1}{t^{\frac{3p+1}{2}}}\right),~\textup{as} ~ t\to +\infty,
\end{equation*}
\begin{equation*}
|f(x(t))-f(x^*)|=O\left( \frac{1}{t^{\frac{3p+1}{2}}}\right),~\textup{as} ~ t\to +\infty,
\end{equation*}
and
\begin{equation*}
\int_{t_0}^{+\infty}[\dot{\tau}(s)]^2 \left\| \nabla_x\mathcal{L}\left(x(s),\lambda(s)\right) \right\|^2 ds <+\infty.
\end{equation*}
Furthermore, $(x(s),\lambda(s))$ is bounded over $[t_0,t]$, and the upper bound depends only on the initial condition.
\end{corollary}

\begin{remark}
Note that Lin and Michael \cite{lin2022} introduced the second-order dynamical system with closed-loop control (\ref{Lin}).
In \cite[Theorem 3]{lin2022}, they obtained
$$
f(x(t))-f(x^*)=O\left(\frac{1}{t^{\frac{3p+1}{2}}}\right),~\textup{as} ~ t\to +\infty.
$$
Therefore, Corollary \ref{Co2} extends \cite[Theorem 3]{lin2022} from the unconstrained optimization problem to Problem (\ref{non}).
\end{remark}

In the special case of System (\ref{dyn}) where $p=q=1$, we have $\mu(t)=1$ and $\tau(t)=t$. Then, System (\ref{dyn}) collapses to the following open-loop system:
\begin{equation}\label{h9}
\begin{cases}
\ddot{x}(t)+\frac{2}{t}\dot{x}(t)+\frac{1}{t}
\frac{d}{dt}\nabla_{x}\mathcal{L}\left(x(t),\lambda(t)\right)
+\frac{1}{t}\nabla_{x}\mathcal{L}\left(x(t),\lambda(t)\right) = 0,\\
\dot{\lambda}(t)-\nabla_\lambda \mathcal{L}\left(x(t)+t \dot{x}(t),\lambda(t)\right)=0.
\end{cases}
\end{equation}

\begin{corollary}\label{Co3}
Let $ (x ,\lambda ):\left[t_0, +\infty\right)\to \mathcal{H} \times \mathcal{G}$ be a solution of System $ (\ref{h9})$. Then, for any $ (x^*,\lambda^*)\in \mathcal{S} $, it holds that
\begin{equation*}
\mathcal{L}(x(t),\lambda^*)-\mathcal{L}(x^*,\lambda^*)=O\left(\frac{1}{t}\right),~\textup{as} ~  t\to +\infty ,
\end{equation*}
\begin{equation*}
\|Ax(t)-b\|=O\left( \frac{1}{t}\right),~\textup{as} ~ t\to +\infty,
\end{equation*}
\begin{equation*}
|f(x(t))-f(x^*)|=O\left( \frac{1}{t}\right),~\textup{as} ~ t\to +\infty,
\end{equation*}
and
\begin{equation*}
\int_{t_0}^{+\infty} \left\| \nabla_x\mathcal{L}\left(x(s),\lambda(s)\right) \right\|^2 ds <+\infty.
\end{equation*}
Furthermore, $(x(s),\lambda(s))$ is bounded over $[t_0,t]$, and the upper bound depends only on the initial condition.
\end{corollary}

\begin{remark}
System (\ref{h9}) can be seen as a special case of the system proposed in \cite{hetian2025}. Clearly, the parameters of System (\ref{h9}) satisfy the Assumption 2.1 in \cite{hetian2025}. Thus, by Theorem 2.3 in \cite{hetian2025}, we have
\begin{equation*}
\mathcal{L}(x(t),\lambda^*)-\mathcal{L}(x^*,\lambda^*)=O\left(\frac{1}{t}\right),~\textup{as} ~  t\to +\infty ,
\end{equation*}
\begin{equation*}
\|Ax(t)-b\|=O\left( \frac{1}{t}\right),~\textup{as} ~ t\to +\infty,
\end{equation*}
and
\begin{equation*}
|f(x(t))-f(x^*)|=O\left( \frac{1}{t}\right),~\textup{as} ~ t\to +\infty,
\end{equation*}
which is compatible with the results in Corollary \ref{Co3}.
\end{remark}

\section{The discrete accelerated autonomous primal-dual algorithm}
In this section, we propose an accelerated autonomous primal-dual algorithm with adaptive step size defined via gradient and analyze the convergence properties of this algorithm.
For simplicity, we restrict our discussion to the case where $q = 1$, which gives $\dot{\tau} (t) = \mu (t)$.
Obviously, System (\ref{dy1}) can be reformulated as:
\begin{equation}\label{h10}
\begin{cases}
y(t)=x(t)+ \frac{\tau(t)}{\dot{\tau}(t)}\dot{x}(t),\\
\dot{y}(t)=-\dot{\tau}(t)
\frac{d}{dt}\nabla_{x}\mathcal{L}\left(x(t),\lambda(t)\right)-(\dot{\tau}(t)+\ddot{\tau}(t))\nabla_{x}\mathcal{L}\left(x(t),\lambda(t)\right),\\
\dot{\lambda}(t)=\dot{\tau}(t)\left( Ay(t)-b \right),\\
\dot{\tau}(t)=\mu(t),\\
[\mu(t)]^p \left\| \nabla_{x}\mathcal{L}\left(x(t),\lambda(t)\right) \right\|^{p-1}=1.
\end{cases}
\end{equation}
For System (\ref{h10}), we consider a time step size of $1$ and set $y(k)\approx y_k$, $x(k)\approx x_k$, $\tau(k) \approx \tau_k$, $\dot{\tau}(k)\approx \gamma_{k}$,  $\lambda(k)\approx \lambda_{k}$, and $\mu(k)\approx \mu_k $. We evaluate the terms $\frac{d}{dt}\nabla_{x}\mathcal{L}\left(x(t),\lambda(t)\right)$ and $\ddot{\tau}(t)$, which are the time derivative of $\nabla_{x}\mathcal{L}\left(x(t),\lambda(t)\right)$ and $\dot{\tau}(t)$, respectively. Now, by time discretization of System (\ref{h10}), we obtain
\begin{equation}\label{h12}
\begin{cases}
y_k= x_k+ \frac{\tau_k}{\gamma_{k}}(x_k-x_{k-1}), \\
y_{k+1}-y_k = -\gamma_{k}
        \left[\nabla f(x_{k+1})+A^{*} \lambda_{k+1}-(\nabla f(x_{k})+A^{*} \lambda_{k}) \right] \\
\qquad\qquad\qquad-(\gamma_{k+1}+\gamma_{k+1}-\gamma_{k})\left( \nabla f(x_{k+1})+A^{*} \lambda_{k+1} \right), \\
\lambda_{k+1}-\lambda_k = \gamma_{k+1}\left( Ay_{k+1}-b \right), \\
\tau_{k+1} = \gamma_{k}+\tau_k, \\
\gamma_{k+1} = \mu_{k}, \\
\mu_{k}^p \left\| \nabla_{x}\mathcal{L}\left(x_k,\lambda_k\right) \right\|^{p-1} = 1.
\end{cases}
\end{equation}
Substituting the first equation and the third equation into the second equation in System (\ref{h12}), we have
\begin{equation*}
\begin{aligned}
  \nabla f(x_{k+1})+\frac{\gamma_{k+1}+\tau_{k+1}}{2 \gamma_{k+1}^2 }(x_{k+1}-\bar{x}_k) +(\gamma_{k+1}+\tau_{k+1}) A^*( Ax_{k+1}- \sigma_{k+1})=0,
\end{aligned}
\end{equation*}
where
\begin{equation*}
\bar{x}_k :=x_k+\frac{\gamma_{k+1}}{\gamma_{k+1}+\tau_{k+1}}\left(\frac{\tau_k}{\gamma_{k}}(x_k-x_{k-1})+\gamma_{k} \left(\nabla f(x_k)+A^* \lambda_k \right)\right),
\end{equation*}
and
\begin{equation*}
\sigma_{k+1}:= \frac{1}{\gamma_{k+1}+\tau_{k+1}}\left( \tau_{k+1}Ax_k+\gamma_{k+1}b-\lambda_k \right).
\end{equation*}
This implies that
\begin{equation*}
\begin{aligned}
  x_{k+1}=&\mathop{\arg\min}\limits_{x\in \mathcal{H}} \left\{ f(x)+\frac{\gamma_{k+1}+\tau_{k+1}}{4 \gamma_{k+1}^2 }\| x-\bar{x}_k \|^2 +\frac{ \gamma_{k+1}+\tau_{k+1} }{2} \left\| Ax- \sigma_{k+1} \right\|^2 \right\}.
\end{aligned}
\end{equation*}
Based on the above analysis, we are now in a position to introduce the following accelerated autonomous primal-dual algorithm for solving Problem (\ref{constrained}).

\begin{algorithm}[H]
\caption{Accelerated autonomous primal-dual algorithm (AAPDA)}
\textbf{Initialization:} Choose  $x_0 = x_1\in \mathcal{H} \textup{ and } \lambda_1\in \mathcal{G}$. Let $p>1$, $\tau_1= 0$ and $\gamma_1\geq 1$.
\begin{algorithmic}
\For{$k = 1, 2, \dots$}
    \If{$$
    \nabla_{x}\mathcal{L}\left(x_k,\lambda_k\right)=0,
    $$}
    \Stop
    \Else
    \State \textbf{Step 1:} Compute
  \begin{equation}\label{qq2}
  \mu_{k} =\left\|\nabla_{x}\mathcal{L}\left(x_k,\lambda_k\right)\right\|^{-\frac{p-1}{p}},
  \end{equation}
  \begin{equation}\label{qq3}
  \gamma_{k+1}=\mu_{k},
  \end{equation}
  \begin{equation}\label{qq1}
  \tau_{k+1}=\gamma_{k}+\tau_k,
  \end{equation}
  \begin{equation}\label{q2}
  \bar{x}_k :=x_k+\frac{\gamma_{k+1}}{\gamma_{k+1}+\tau_{k+1}}\left(\frac{\tau_k}{\gamma_{k}}(x_k-x_{k-1})+\gamma_{k} \left(\nabla f(x_k)+A^* \lambda_k \right)\right),
  \end{equation}
  \begin{equation}\label{q4}
  \sigma_{k+1}:= \frac{1}{\gamma_{k+1}+\tau_{k+1}}\left( \tau_{k+1}Ax_k+\gamma_{k+1}b-\lambda_k \right).
  \end{equation}
    \State \textbf{Step 2:} Update the primal variable
  \begin{equation}\label{q8}\small
  \begin{aligned}
  x_{k+1}=&\mathop{\arg\min}\limits_{x\in \mathcal{H}} \left\{ f(x)+\frac{\gamma_{k+1}+\tau_{k+1}}{4 \gamma_{k+1}^2 }\| x-\bar{x}_k \|^2 +\frac{ \gamma_{k+1}+\tau_{k+1} }{2} \left\| Ax- \sigma_{k+1} \right\|^2 \right\}.
  \end{aligned}
  \end{equation}

    \State \textbf{Step 3:} Compute
    \begin{equation}\label{q9}
    y_{k+1} = x_{k+1}+\frac{\tau_{k+1}}{\gamma_{k+1}}(x_{k+1}-x_k),
    \end{equation}
  and update the dual variable
  \begin{equation}\label{q10}
  \lambda_{k+1}=\lambda_k+\gamma_{k+1} (Ay_{k+1}-b).
  \end{equation}
    \EndIf
\EndFor
\end{algorithmic}
\end{algorithm}

In order to analyze the convergence rates of AAPDA, we need the following result.

\begin{lemma}\textup{\cite[Lemma 4.2]{attouchbot2025}}\label{lemma2}
Let $\{\mu_k\}_{k \geq 0}$ be a positive sequence and let $\{\tau_k\}_{k \geq 1}$ be a sequence such that $\tau_1=0$ and $\tau_k=\sum_{i=0}^{k-2}\mu_i$  for all $k \geq 2$. Suppose that there exist constants $C_4>0$ and $a,b,c \geq 0$ such that $b+c>a$ and
\begin{equation*}
\sum_{k \geq 0} \tau_k^a \mu_{k}^{-b} \mu_{k+1}^{-c} \leq C_4 <+\infty.
\end{equation*}
Then, there exists a constant $C_5 >0$ such that for any $k \geq 2$, it holds that
$
\tau_{k+1} \geq C_5 k^{\frac{b+c+1}{b+c-a}}.
$
\end{lemma}

The following equality will be used in the sequel:
\begin{equation}\label{o1}
\frac{1}{2}\|x_1\|^2-\frac{1}{2}\|x_2\|^2=\langle x_1,x_1-x_2 \rangle-\frac{1}{2} \|x_1-x_2\|^2,~~ \forall x_1,x_2\in \mathcal{H}.
\end{equation}

We now establish the convergence rates of the primal-dual gap, the feasibility violation, and the objective residual generated by AAPDA.

\begin{theorem}\label{theorem4.1}
Let $\{(x_k,\lambda_k)\}_{k \geq 0}$ be the sequence generated by $\rm{AAPDA}$. Suppose that $\{\mu_k\}_{k \geq 0}$ is a non-decreasing sequence satisfying $\mu_k \geq 1$  for all $ k \geq 0$. Also, let $\{\tau_k\}_{k \geq 1}$  be a sequence  such that $\tau_1=0$ and   $\tau_k = \sum_{i=0}^{k-2}\mu_i$  for all $k \geq 2$. Then, for any $\left( x^*,\lambda^* \right)\in \mathcal{S} $, it holds that
$$
\mathcal{L}(x_k, \lambda^*)-\mathcal{L}(x^*, \lambda^*)=O\left( \frac{1}{k^{\frac{3p-1}{2p}}} \right), \textup{ as } k\rightarrow +\infty,
$$
$$
\| A x_k -b \| = O\left( \frac{1}{k^{\frac{3p-1}{2p}}} \right), \textup{ as } k\rightarrow +\infty,
$$

$$
| f(x_k) -f(x^*) | = O\left( \frac{1}{k^{\frac{3p-1}{2p}}} \right), \textup{ as } k\rightarrow +\infty,
$$
and
$$
\sum_{k=0}^{+\infty} \gamma_{k+1}^2 \lVert \nabla_x \mathcal{L}(x_{k+1},\lambda_{k+1}) \rVert^2 < +\infty.
$$
\end{theorem}

\begin{proof}
For $\left(x^*, \lambda^* \right)\in \mathcal{S}$, we introduce the following energy sequence
\begin{equation*}\label{bb1}
E_k:=\tau_{k+1} \left(\mathcal{L}(x_k,\lambda^*)-\mathcal{L}(x^*,\lambda^*)\right)
+\frac{1}{2}\lVert u_k \rVert^2
+\frac{1}{2}\lVert \lambda_{k}-\lambda^*\rVert^2,
\end{equation*}
where $u_k:=y_{k}-x^*+\gamma_{k} \nabla_x \mathcal{L}(x_k,\lambda_k)$.

Clearly, from (\ref{q2}), (\ref{q4}), (\ref{q8}), (\ref{q9}) and (\ref{q10}) (see also the second equation of   system (\ref{h12})), it is easy to show that
\begin{equation*}
y_{k+1}-y_k = \gamma_{k} \nabla_x \mathcal{L}(x_k,\lambda_k)
-2\gamma_{k+1} \nabla_x \mathcal{L}(x_{k+1},\lambda_{k+1}).
\end{equation*}
Then,
\begin{equation*}\label{h11}
\begin{aligned}
u_{k+1}-u_k &= y_{k+1}-y_k+\gamma_{k+1} \nabla_x \mathcal{L}(x_{k+1},\lambda_{k+1})-\gamma_{k} \nabla_x \mathcal{L}(x_k,\lambda_k)\\
&=-\gamma_{k+1} \nabla_x \mathcal{L}(x_{k+1},\lambda_{k+1}).
\end{aligned}
\end{equation*}
This together with (\ref{o1}) and $\nabla_x \mathcal{L}\left(x_{k+1},\lambda_{k+1}\right)=\nabla_x \mathcal{L}\left(x_{k+1},\lambda^*\right)+A^*(\lambda_{k+1}-\lambda^*)$ gives
\begin{equation}\label{b1}
\begin{aligned}
&\frac{1}{2}\lVert u_{k+1} \rVert^2-\frac{1}{2}\lVert u_{k}\rVert^2\\
=& \langle u_{k+1}-u_k, u_{k+1} \rangle-\frac{1}{2} \lVert u_{k+1}-u_k \rVert^2\\
\leq& -\gamma_{k+1}\langle \nabla_x \mathcal{L}(x_{k+1},\lambda_{k+1}), y_{k+1}-x^* \rangle-\gamma_{k+1}^2 \lVert \nabla_x \mathcal{L}(x_{k+1},\lambda_{k+1}) \rVert^2\\
=&-\gamma_{k+1}\langle \nabla_x \mathcal{L}(x_{k+1},\lambda^*), y_{k+1}-x^* \rangle
-\gamma_{k+1}\langle A^*(\lambda_{k+1}-\lambda^*), y_{k+1}-x^* \rangle\\
&-\gamma_{k+1}^2 \lVert \nabla_x \mathcal{L}(x_{k+1},\lambda_{k+1}) \rVert^2.
\end{aligned}
\end{equation}
From (\ref{q9}) and the convexity of $\mathcal{L}(\cdot,\lambda^*)$, we obtain
\begin{equation}\label{k1}
\begin{aligned}
&-\gamma_{k+1}\langle \nabla_x \mathcal{L}(x_{k+1},\lambda^*), y_{k+1}-x^* \rangle\\
=&-\gamma_{k+1}\langle \nabla_x \mathcal{L}(x_{k+1},\lambda^*), x_{k+1}-x^* \rangle
-\tau_{k+1}\langle \nabla_x \mathcal{L}(x_{k+1},\lambda^*), x_{k+1}-x_k \rangle\\
\leq& -\gamma_{k+1} (\mathcal{L}(x_{k+1},\lambda^*)-\mathcal{L}(x^*,\lambda^*))
-\tau_{k+1} (\mathcal{L}(x_{k+1},\lambda^*)-\mathcal{L}(x_{k},\lambda^*))\\
=&-(\gamma_{k+1}+\tau_{k+1})(\mathcal{L}(x_{k+1},\lambda^*)-\mathcal{L}(x^*,\lambda^*))
+\tau_{k+1}(\mathcal{L}(x_{k},\lambda^*)-\mathcal{L}(x^*,\lambda^*)).
\end{aligned}
\end{equation}
On the other hand, it follows from (\ref{q10}) and (\ref{o1}) that
\begin{equation}\label{o2}\small
\begin{aligned}
-\gamma_{k+1}\langle A^*(\lambda_{k+1}-\lambda^*), y_{k+1}-x^* \rangle
&= -\gamma_{k+1}\langle \lambda_{k+1}-\lambda^*, Ay_{k+1}-b \rangle\\
&= -\langle \lambda_{k+1}-\lambda^*, \lambda_{k+1}-\lambda_k \rangle\\
&=-\frac{1}{2}\left( \lVert \lambda_{k+1}-\lambda^* \rVert^2-\lVert \lambda_{k}-\lambda^* \rVert^2+\lVert \lambda_{k+1}-\lambda_{k} \rVert^2 \right).
\end{aligned}
\end{equation}
Combining (\ref{b1})--(\ref{o2}), it gives
\begin{equation}\label{b6}
\begin{aligned}
E_{k+1}-E_{k}&=\tau_{k+2} \left(\mathcal{L}(x_{k+1},\lambda^*)-\mathcal{L}(x^*,\lambda^*)\right)
-\tau_{k+1} \left(\mathcal{L}(x_k,\lambda^*)-\mathcal{L}(x^*,\lambda^*)\right)\\
&~~~~+\frac{1}{2}\lVert u_{k+1} \rVert^2-\frac{1}{2}\lVert u_{k} \rVert^2
+\frac{1}{2}\lVert \lambda_{k+1}-\lambda^*\rVert^2-\frac{1}{2}\lVert \lambda_{k}-\lambda^*\rVert^2\\
&\leq (\tau_{k+2}-\tau_{k+1}-\gamma_{k+1})\left(\mathcal{L}(x_{k+1},\lambda^*)-\mathcal{L}(x^*,\lambda^*)\right)\\
&~~~~-\gamma_{k+1}^2 \lVert \nabla_x \mathcal{L}(x_{k+1},\lambda_{k+1}) \rVert^2
-\frac{1}{2}\lVert \lambda_{k+1}-\lambda_{k} \rVert^2\\
&\leq 0,
\end{aligned}
\end{equation}
where the last inequality holds due to (\ref{qq1}). Thus, the sequence $\{ E_k \}_{k\geq 0}$ is non-increasing. This implies that
\begin{equation}\label{cc1}
\mathcal{L}(x_k, \lambda^*)-\mathcal{L}(x^*, \lambda^*)=O\left( \frac{1}{\tau_{k+1}} \right), \textup{ as } k\rightarrow +\infty,
\end{equation}
and $\{\lambda_k\}_{k\geq 0}$ is bounded.

For notational simplicity, denote a sequence $\{g_k\}_{k\geq 0}$ by
\begin{equation*}
g_k:=\tau_{k+1}(Ax_k-b).
\end{equation*}
It follows from (\ref{qq1}), (\ref{q9}) and (\ref{q10}) that
\begin{equation*}
\begin{aligned}
\lambda_{k+1}-\lambda_0&=\sum_{j=0}^k (\lambda_{j+1}-\lambda_j)\\&=\sum_{j=0}^k \gamma_{j+1}(Ay_{j+1}-b)\\
&=\sum_{j=0}^k \left[ (\gamma_{j+1}+\tau_{j+1})(Ax_{j+1}-b)-\tau_{j+1}(Ax_{j}-b) \right]\\
&=\sum_{j=0}^k (g_{j+1}-g_j)\\
&=g_{k+1}-g_0.
\end{aligned}
\end{equation*}
From this, together with the boundedness of $\{\lambda_k\}_{k\geq 0}$, it follows that $\sup_{k\geq 0} \|g_k\| < +\infty$. Then,
\begin{equation}\label{cc2}
\|Ax_k-b\| \leq O\left( \frac{1}{\tau_{k+1}} \right), \textup{ as } k\rightarrow +\infty.
\end{equation}
Based on (\ref{cc1}), (\ref{cc2}) and the definition of $\mathcal{L}$, we obtain
\begin{equation}\label{cc3}
\left| f(x_k)-f(x^*) \right| \leq \mathcal{L}(x_k,\lambda^*)-\mathcal{L}(x^*,\lambda^*)+\|\lambda^*\| \|Ax_k-b\|\leq O\left( \frac{1}{\tau_{k+1}} \right), \textup{ as } k\rightarrow +\infty.
\end{equation}
Summing (\ref{b6}) over $k=0,\dots,N$, we obtain
\begin{equation*}
\sum_{k=0}^N \gamma_{k+1}^2 \lVert \nabla_x \mathcal{L}(x_{k+1},\lambda_{k+1}) \rVert^2
\leq E_0-E_N\leq E_0.
\end{equation*}
Letting $N\rightarrow +\infty$ in the above inequality, we obtain
\begin{equation*}
\sum_{k=0}^{+\infty} \gamma_{k+1}^2 \lVert \nabla_x \mathcal{L}(x_{k+1},\lambda_{k+1}) \rVert^2
< +\infty.
\end{equation*}
This together with (\ref{qq2}) and (\ref{qq3}) gives
$$
\sum_{k=0}^{+\infty} \mu_{k}^2 \mu_{k+1}^{-\frac{2p}{p-1}} < +\infty.
$$
Since $\mu_k \geq 1$  for all $k \geq 0$, we have
$$
\sum_{k=0}^{+\infty} \mu_{k+1}^{-\frac{2p}{p-1}}\leq \sum_{k=0}^{+\infty} \mu_{k}^2 \mu_{k+1}^{-\frac{2p}{p-1}} < +\infty.
$$
Thus, by using Lemma \ref{lemma2}, it follows that there exists a constant $C_3>0$ such that
$$\tau_{k+1} \geq C_3 k^{\frac{3p-1}{2p}}.$$ This together with (\ref{cc1}), (\ref{cc2}) and (\ref{cc3}) gives
$$
\mathcal{L}(x_k, \lambda^*)-\mathcal{L}(x^*, \lambda^*)=O\left( \frac{1}{k^{\frac{3p-1}{2p}}} \right), \textup{ as } k\rightarrow +\infty,
$$
$$
\| A x_k -b \| = O \left( \frac{1}{k^{\frac{3p-1}{2p}}} \right), \textup{ as } k\rightarrow +\infty,
$$
and
$$
| f(x_k) -f(x^*) | = O \left( \frac{1}{ k^{\frac{3p-1}{2p}}} \right), \textup{ as } k\rightarrow +\infty.
$$
The proof is complete.\qed
\end{proof}

\section{Numerical experiments}
In this section,  we validate the theoretical findings of previous sections through numerical experiments. In these experiments, all code is implemented in MATLAB R2021b and executed on a PC equipped with a 2.40GHz Intel Core i5-1135G7 processor and 16GB  of RAM.

\begin{example}\label{example5.2}
Consider the following constrained minimization problem:
\begin{eqnarray*}
\left\{ \begin{array}{ll}
&\mathop{\mbox{min}}\limits_{x\in \mathbb{R}^n} \frac{\mu}{2} \|x\|^2\\
&\mbox{s.t.}~~Ax=b,
\end{array}
\right.
\end{eqnarray*}
where $\mu\geq 0$, $ A\in \mathbb{R}^{m\times n} $ and $ b\in\mathbb{R}^m $. We generate the matrix $A$ using the standard Gaussian distribution. The original solution $x^*\in\mathbb{R}^n$ is generated from a Gaussian distribution $\mathcal{N}(0,4)$, with its entries clipped to the interval [-2,2] and sparsified so that only $1\%$ of its elements are non-zero. Furthermore, we select $b=Ax^*$.

\vspace{-1em}
 \begin{figure}[H]%
     \centering
         \subfloat[$  n=m=10$ ]{
         \includegraphics[width=0.48\linewidth]{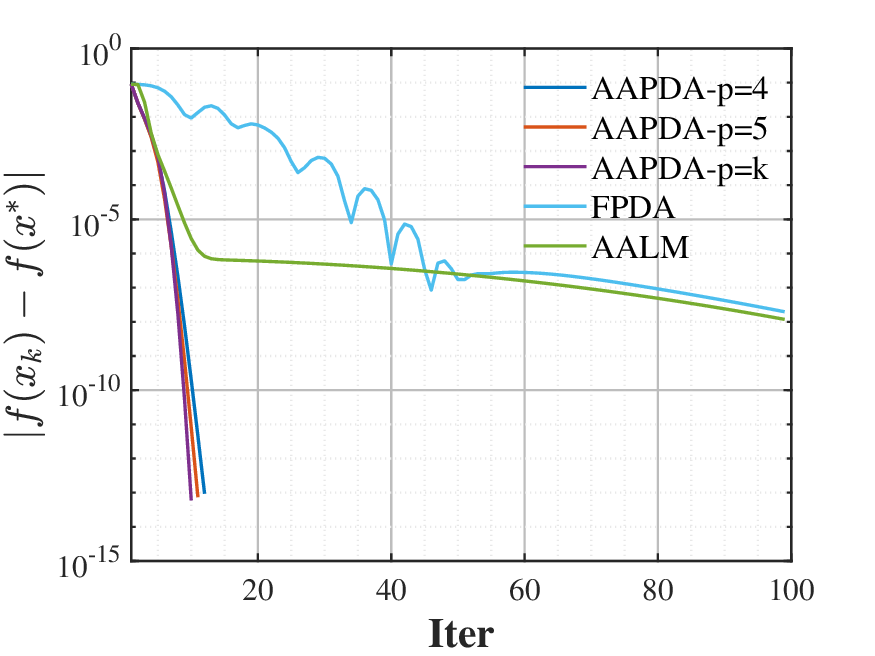}
         \hfill
         \includegraphics[width=0.48\linewidth]{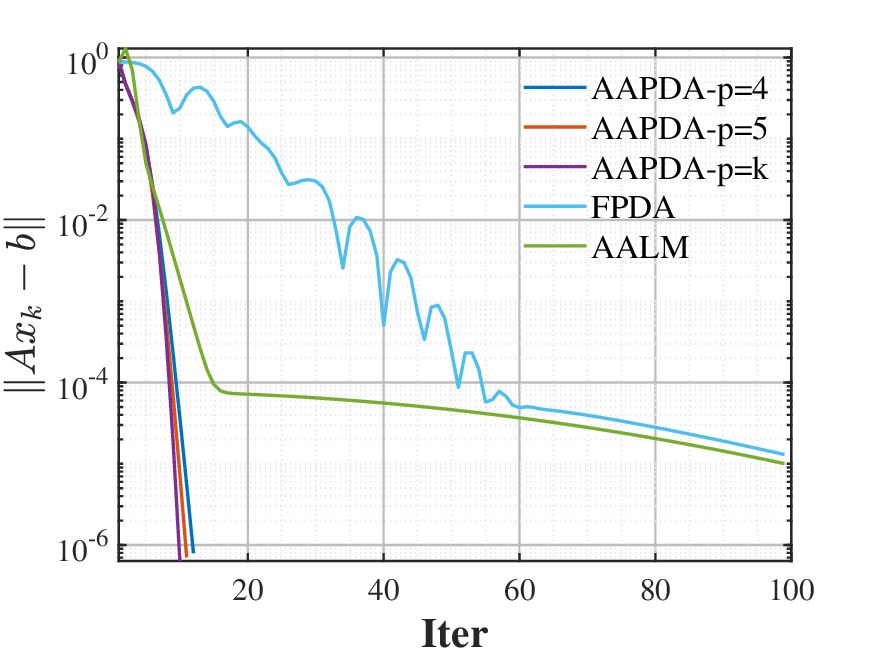}
         }\\
     \subfloat[$  n=m=300 $ ]{
         \includegraphics[width=0.48\linewidth]{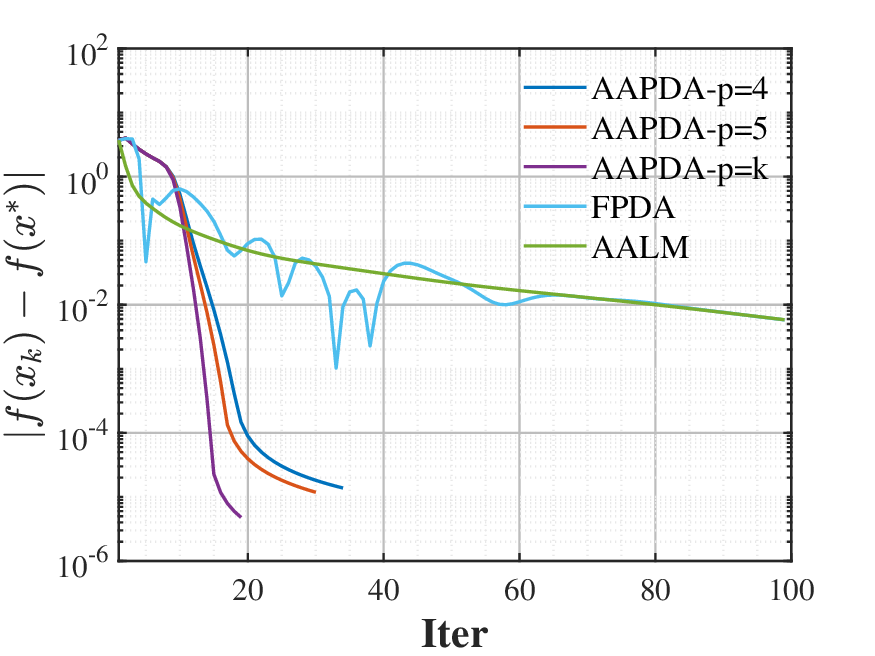}
         \hfill
         \includegraphics[width=0.48\linewidth]{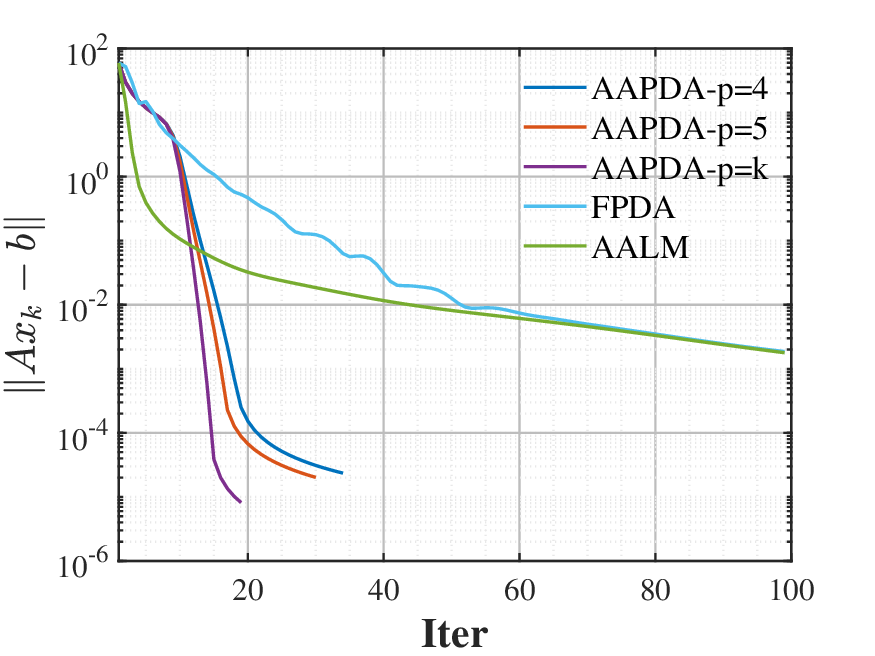}
         }\\
      \subfloat[$  n=m=2000 $ ]{
         \includegraphics[width=0.48\linewidth]{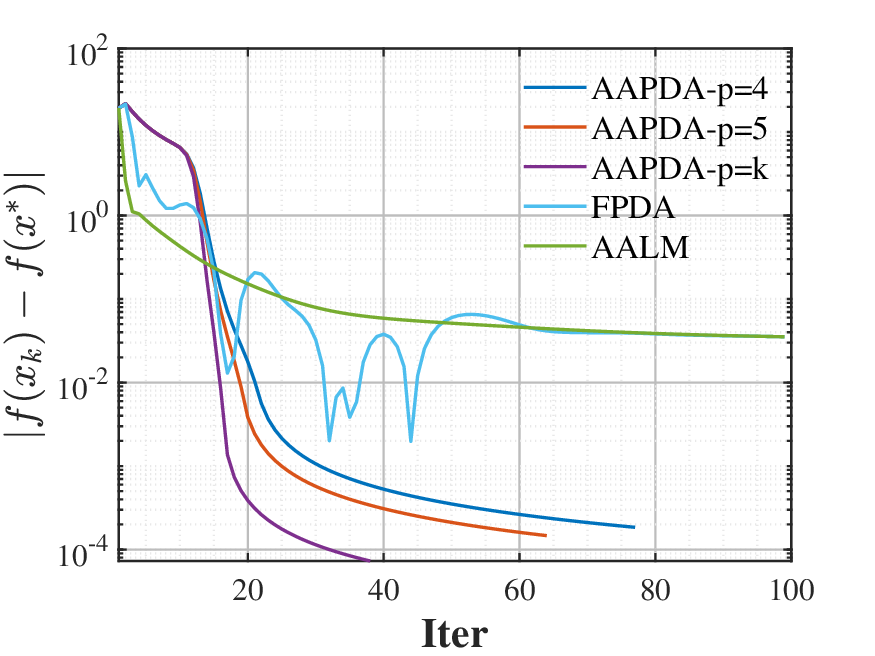}
         \hfill
         \includegraphics[width=0.48\linewidth]{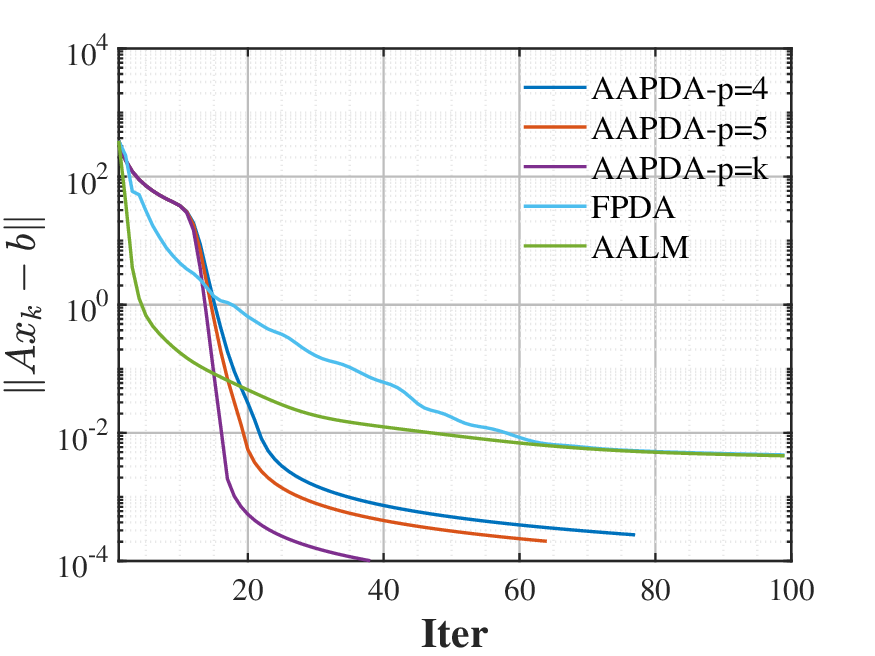}
         }\\
     \caption{Numerical results obtained using AAPDA, FPDA and AALM under different dimensions.}
     \label{HD-exp2}
 \end{figure}
 \vspace{-1em}

Let $\mu=1.5$. The stopping condition is
$$
\frac{ \left\|x_{k+1}-x_{k} \right\| }{ \max \{ \left\|x_{k} \right\|,1\} } \leq 10^{-6}
$$
or the number of iterations exceeds 100.

In the following experiments, we compare AAPDA with the fast primal-dual algorithm (FPDA) proposed in \cite{hehu2022} and the accelerated linearized augmented Lagrangian method (AALM) proposed in \cite{xu2017}. Here are the parameter settings for the algorithms:
\begin{itemize}
\item AAPDA:  $ \gamma_1=1$, $\tau_1=0$ and $p=\{4,5,k\}$.
\item FPDA: $\alpha=50$, $\theta=2$, $\beta_0=\frac{0.2}{\theta}$ and $\epsilon_k=0$.
\item AALM: $\gamma=0.1$, $\alpha_k=\frac{2}{k+1}$, $\beta_k=\gamma_k=k\gamma$ and $P_k=\frac{1}{k}\mathrm{Id}$.
\end{itemize}

As shown in Figure \ref{HD-exp2}, AAPDA achieves significantly superior performance over FPDA and AALM under all three different dimensions. It is worth noting that the convergence speed and accuracy of AAPDA improve as the value of $p$ increases, and that AAPDA can achieve exponential convergence. Moreover, compared to FPDA, the Hessian-driven damping term in AAPDA can induces significant attenuation of the oscillations.

\end{example}

\begin{example}\label{example5.3}
Consider the following non-negative least squares problem:
\begin{eqnarray*}
\mathop{\mbox{min}}\limits_{x\in \mathbb{R}^n} f(x):=\frac{1}{2} \|Ax-b\|^2,
\end{eqnarray*}
where $ A\in \mathbb{R}^{m\times n} $ and $ b\in\mathbb{R}^m $. We generate a random matrix $A\in \mathbb{R}^{m \times n}$ with density $s\in (0,1]$ and a random vector $ b\in\mathbb{R}^m $. The nonzero entries of $A$ are independently generated from a uniform distribution on $[0,0.1]$. The true solution $x^*$ is obtained via $x^*: = A \setminus b$ in MATLAB.

Let $m=500$ and $n=1000$. The stopping condition is
$$
\frac{ \left\|x_{k+1}-x_{k} \right\| }{ \max \{ \left\|x_{k} \right\|,1\} } \leq \theta
$$
or the number of iterations exceeds 200. For $s=0.5$ and $s=1$, we consider $\theta:=\{10^{-6},10^{-8},\\
10^{-10}\}$, respectively.

In the following experiments, we compare AAPDA with FISTA proposed in \cite{beck2009}, the proximal inertial algorithm (PIA) proposed in \cite[Algorithm 4]{attouchbot2025} and the accelerated forward-backward method (AFBM) proposed in \cite{a2016}. Here are the parameter settings for the algorithms:
\begin{itemize}
\item AAPDA:  $ \gamma_1=5$, $\tau_1=0$ and $p=5$.
\item FISTA: $\alpha=\frac{1}{\|A\|^2}$ and $t_1=1$.
\item PIA: $p=5$ and $\tau_1=0$.
\item AFBM: $\alpha=\frac{1}{\|A\|^2}$, $t_1=1$ and $\beta=5$.
\end{itemize}

\vspace{-1em}
 \begin{figure}[H]%
     \centering
         \subfloat[$  s=0.5 $ ]{
         \includegraphics[width=0.48\linewidth]{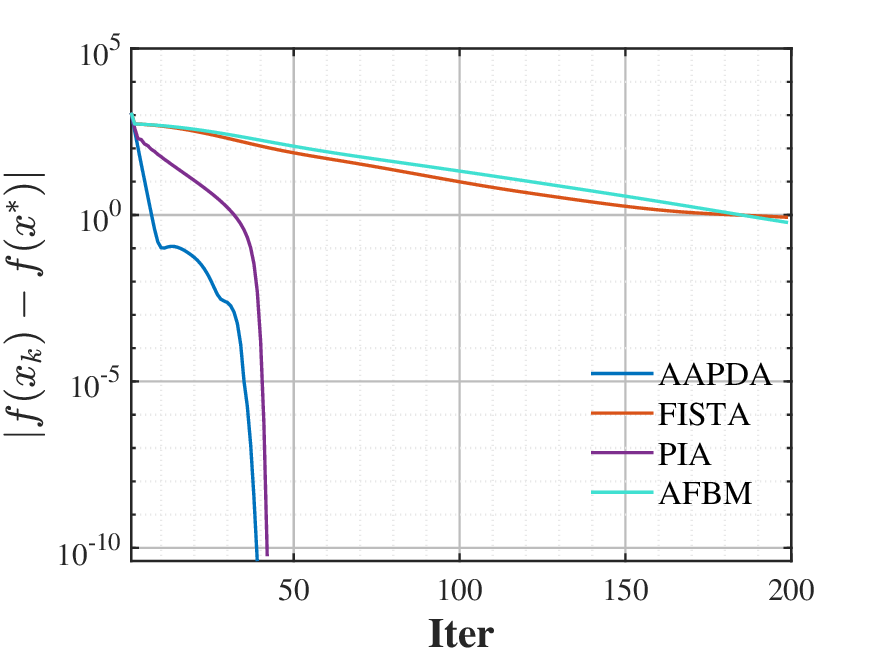}
         }
         \subfloat[$  s=1 $ ]{
         \includegraphics[width=0.48\linewidth]{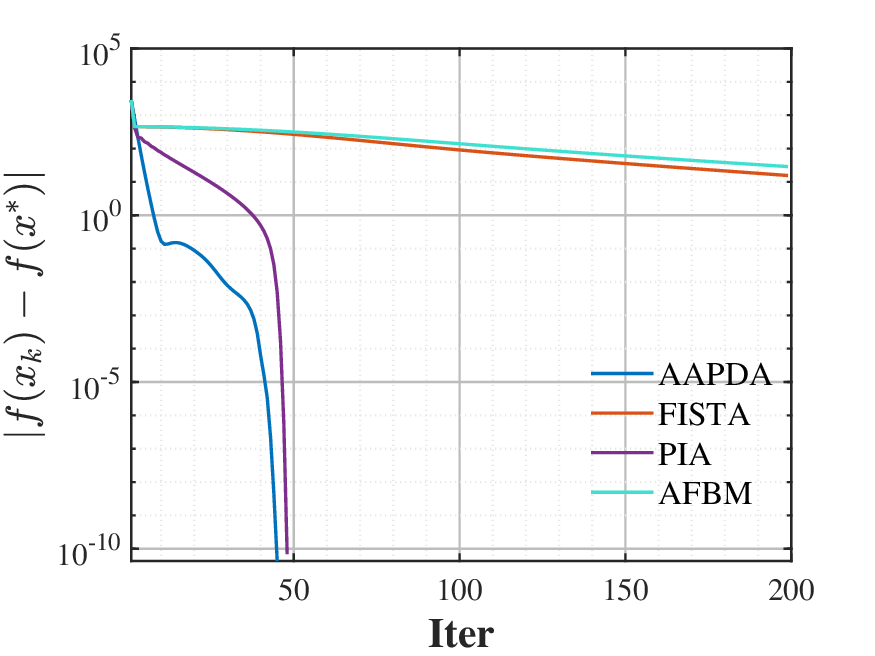}
         }\\
     \caption{Numerical results obtained using AAPDA, FISTA, PIA and AFBM when $\theta=10^{-6}$.}
     \label{HD-exp3}
 \end{figure}
 \vspace{-1em}

\vspace{-1em}
 \begin{figure}[H]%
     \centering
         \subfloat[$  s=0.5 $ ]{
         \includegraphics[width=0.48\linewidth]{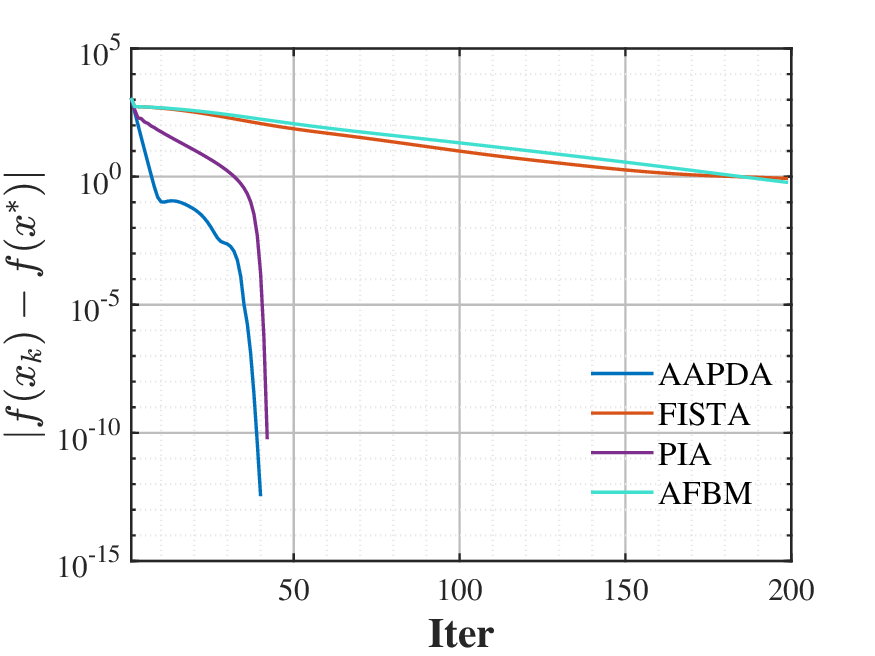}
         }
         \subfloat[$  s=1 $ ]{
         \includegraphics[width=0.48\linewidth]{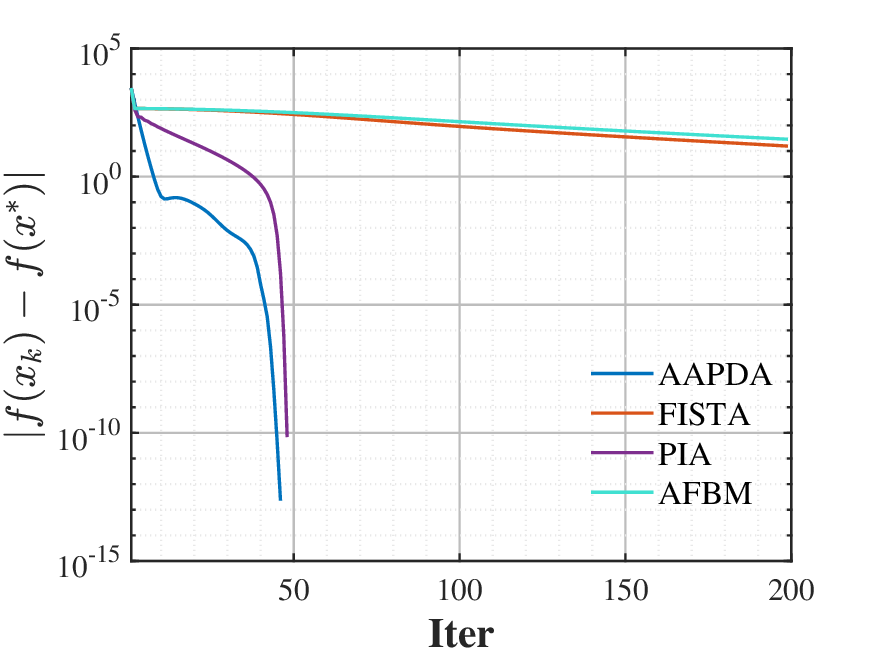}
         }\\
     \caption{Numerical results obtained using AAPDA, FISTA, PIA and AFBM when $\theta=10^{-8}$.}
     \label{HD-exp4}
 \end{figure}
 \vspace{-1em}

 \vspace{-1em}
 \begin{figure}[H]%
     \centering
         \subfloat[$  s=0.5 $ ]{
         \includegraphics[width=0.48\linewidth]{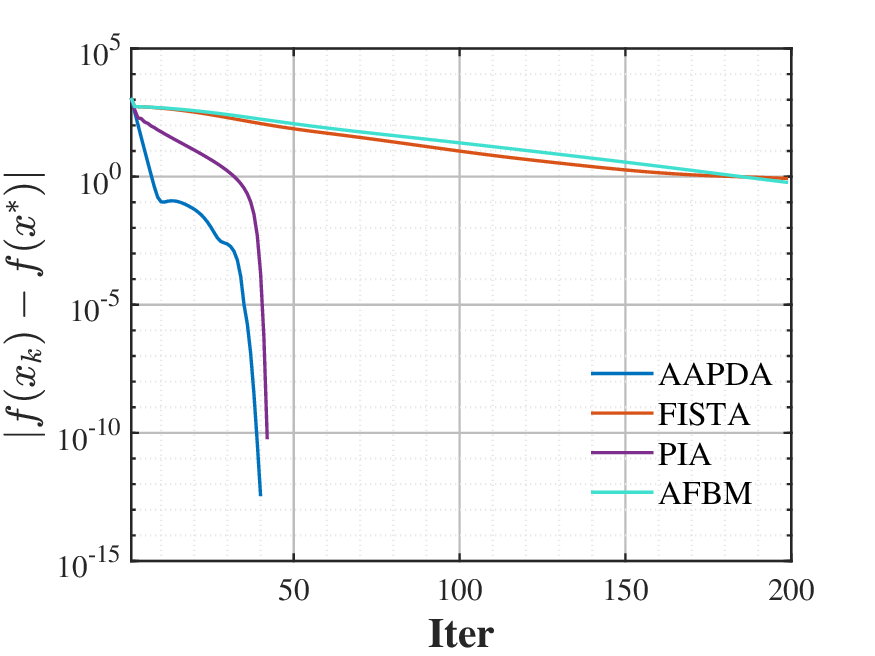}
         }
         \subfloat[$  s=1 $ ]{
         \includegraphics[width=0.48\linewidth]{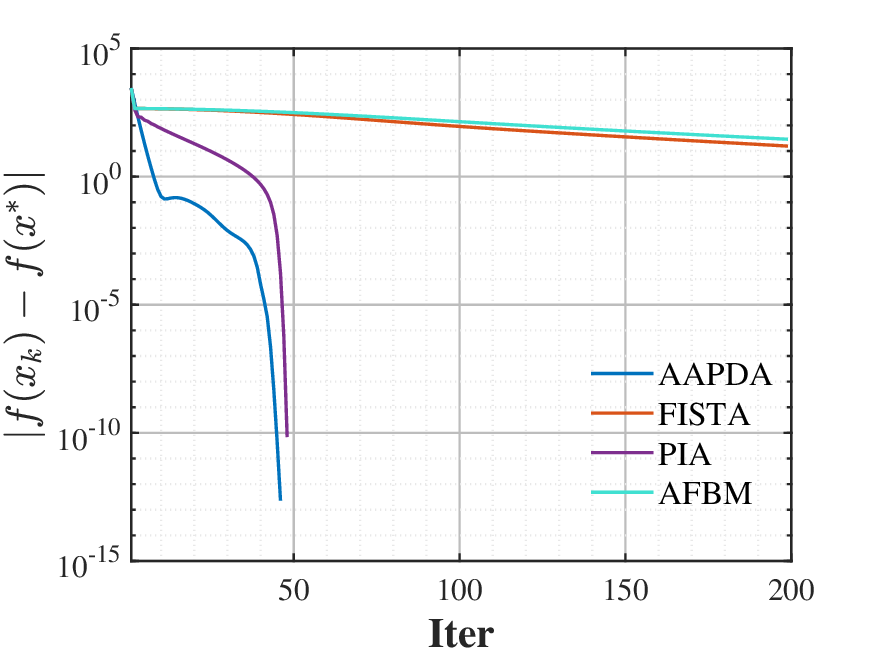}
         }\\
     \caption{Numerical results obtained using AAPDA, FISTA, PIA and AFBM when $\theta=10^{-10}$.}
     \label{HD-exp5}
 \end{figure}
  \vspace{-1em}

As shown in Figures \ref{HD-exp3}, \ref{HD-exp4} and \ref{HD-exp5}, AAPDA demonstrates superior convergence performance compared to FISTA, PIA and AFBM.
Specifically, AAPDA achieves a significantly faster convergence rate and consistently reaches a much higher final accuracy across different values of $\theta$.
Moreover, AAPDA exhibits more stable performance, while the convergence speed and final accuracy of FISTA and AFBM are affected by $\theta$.
\end{example}

\section{Conclusion}
In this paper,  for solving a convex optimization problem with linear equality
constraints (\ref{constrained}), we propose a primal-dual dynamical system (\ref{dyn}), and establish fast convergence rates for the primal-dual gap, the feasibility violation, and the objective residual along its trajectory. A distinctive feature of System (\ref{dyn}) is its damping mechanism, which acts as a feedback control driven by the gradient of the Lagrangian function in Problem (\ref{constrained}).  Building on this continuous-time framework, we introduced an accelerated autonomous primal-dual algorithm (AAPDA) through time discretization of System (\ref{dyn}), and derived analogous convergence rates for the primal-dual gap, the feasibility violation, and the objective residual.

Despite these advances, the convergence analysis of trajectories generated by primal-dual dynamical system with closed-loop control remains an open challenge.
In future work, we aim to address this limitation by reformulating the system and constructing an appropriate energy function, with the goal of rigorously establishing trajectory convergence.

\section*{Funding}
\small{  This research is supported by the Natural Science Foundation of Chongqing (CSTB2024NSCQ-MSX0651) and the Team Building Project for Graduate Tutors in Chongqing (yds223010).}

\section*{Data availability}

 \small{ The authors confirm that all data generated or analysed during this study are included in this article.}

 \section*{Declaration}

 \small{\textbf{Conflict of interest} No potential conflict of interest was reported by the authors.}

\bibliographystyle{plain}

\end{document}